\documentclass[]{spie}  

\usepackage{amsmath,amsfonts,amssymb,amsthm}
\usepackage{graphicx}
\usepackage[colorlinks=true, allcolors=blue]{hyperref}

\usepackage[noabbrev,capitalise]{cleveref}

\newtheorem{thm}{Theorem}
\newtheorem{conj}[thm]{Conjecture}
\newtheorem{prop}[thm]{Proposition}

\crefname{prop}{Proposition}{Propositions}

\title{Linear programming bounds for cliques in Paley graphs}

\author{Mark Magsino}
\author{Dustin G. Mixon}
\author{Hans Parshall}
\affil{Department of Mathematics, The Ohio State University, Columbus, OH 43201}

\authorinfo{Send correspondence to Hans Parshall: E-mail: parshall.6@osu.edu}

\begin{document} 
\maketitle

\begin{abstract}
The Lov\'{a}sz theta number is a semidefinite programming bound on the clique number of (the complement of) a given graph.
Given a vertex-transitive graph, every vertex belongs to a maximal clique, and so one can instead apply this semidefinite programming bound to the local graph.
In the case of the Paley graph, the local graph is circulant, and so this bound reduces to a linear programming bound, allowing for fast computations.
Impressively, the value of this program with Schrijver's nonnegativity constraint rivals the state-of-the-art closed-form bound recently proved by Hanson and Petridis.
We conjecture that this linear programming bound improves on the Hanson--Petridis bound infinitely often, and we derive the dual program to facilitate proving this conjecture.
\end{abstract}


\section{INTRODUCTION}

The \textbf{Paley graph} $G_p$ is defined for every prime $p \equiv 1~(\operatorname{mod}4)$ with vertex set $\mathbf{F}_p = \{0, 1, \ldots, p - 1\}$, the finite field of $p$ elements, and an edge between $x,y \in \mathbf{F}_p$ if and only if $x - y \in Q_p$, where
\[
Q_p := \{x \in \mathbf{F}_p : \text{there exists } y \in \mathbf{F}_p \text{ such that } x = y^2\}
\]
is the multiplicative subgroup of quadratic residues modulo $p$.  The Paley graphs provide a family of quasi-random graphs (see Chung, Graham and Wilson\cite{chung89}) with several nice properties (see \S 13.2 in Bollobas\cite{bollobas01}).  For instance, the Paley graph $G_p$ is a so-called strongly regular graph in which every vertex has $(p - 1)/2$ neighbors, every pair of adjacent vertices share $(p - 5)/4$ common neighbors, and every pair of non-adjacent vertices share $(p - 1)/4$ common neighbors. 
The Paley graph of order $p$ can be used to construct an optimal packing of lines through the origin of $\mathbf{R}^{(p+1)/2}$, known as the corresponding Paley equiangular tight frame~\cite{strohmer03,renes07,waldron09}, and these packings have received some attention in the context of compressed sensing~\cite{bandeira13,bandeira17}.

For a simple, undirected graph $G = (V,E)$, we say $C \subseteq V$ is a \textbf{clique} if every pair of vertices in $C$ is adjacent, and we define the \textbf{clique number} of $G$, denoted by $\omega(G)$, to be the size of the largest clique in $G$.  It is a famously difficult open problem to determine the order of magnitude of $\omega(G_p)$ as $p \rightarrow \infty$.  The best known closed-form bounds that are valid for all primes $p \equiv 1~(\operatorname{mod}4)$ are given by
\begin{equation}\label{eq:bestknown}
(1 + o(1)) \frac{\log(p)}{\log(4)} \leq \omega(G_p) \leq \frac{\sqrt{2p - 1} + 1}{2}.
\end{equation}

The lower bound in \eqref{eq:bestknown} is due to Cohen~\cite{cohen88}.  The same lower bound, with a weaker $o(1)$ term, is actually valid for any self-complementary graph.  Recall that the Ramsey number $R(s)$ is the least integer such that every graph on at least $R(s)$ vertices contains either a clique of size~$s$ or a set of $s$ pairwise non-adjacent vertices.  Then for any self-complementary graph $G$ on at least $R(s)$ vertices, it holds that $\omega(G) \geq s$.  Together with the classical upper bound of $R(s) \leq \binom{2s - 2}{s - 1}$ by Erd\H{o}s and Szekeres~\cite{erdos35}, it is straightforward to establish the lower bound in \eqref{eq:bestknown}.  The work of Graham and Ringrose~\cite{graham90} on least quadratic non-residues shows that there exists $c > 0$ such that $\omega(G_p) \geq c\log(p)\log\log\log(p)$ for infinitely many primes $p$, and so the lower bound in \eqref{eq:bestknown} is not sharp in general.

The upper bound in \eqref{eq:bestknown} was proved very recently by Hanson~and~Petridis~\cite{hanson19} using a clever application of Stepanov's polynomial method.  This improved upon the previously best known closed-form upper bounds of $\omega(G_p) \leq \sqrt{p - 4}$, proved by Maistrelli and Penman~\cite{maistrelli06}, and $\omega(G_p) \leq \sqrt{p} - 1$, proved to hold for a majority of primes $p \equiv 1~(\operatorname{mod}4)$ by Bachoc, Matolcsi and Ruzsa\cite{bachoc13}.  Numerical data for primes $p < 10000$ by Shearer~\cite{shearer86} and Exoo~\cite{exooData} suggests that there should be a polylogarithmic upper bound on $\omega(G_p)$, but it remains an open problem to determine whether there exists $\epsilon > 0$ such that $\omega(G_p) \leq p^{1/2 - \epsilon}$ infinitely often.

This open problem bears some significance in the field of compressed sensing.
In particular, Tao~\cite{taoBlog} posed the problem of finding an explicit family $\{\Phi_n\}$ of $m\times n$ matrices with $m=m(n)\in[0.01n,0.99n]$ and $n\to\infty$ for which there exists $\alpha\geq0.51$ such that for every $n$, it holds that
\[
0.5\cdot\|x\|_2^2
~
\leq
~
\|\Phi_nx\|_2^2
~
\leq
~
1.5\cdot\|x\|_2^2
\]
for every $x\in\mathbb{R}^n$ with at most $n^{\alpha}$ nonzero entries.
Such matrices are known as restricted isometries.
Families of restricted isometries are known to exist for every $\alpha<1$ by an application of the probabilistic method, and yet to date, the best known explicit construction~\cite{bourgain11,mixon15} takes $\alpha\leq\frac{1}{2}+10^{-23}$.
It is conjectured~\cite{bandeira17} that the Paley equiangular tight frame behaves as a restricted isometry for a larger choice of $\alpha$, but proving this is difficult, as it would imply the existence of $\epsilon > 0$ such that $\omega(G_p) \leq p^{1/2 - \epsilon}$ for all sufficiently large $p$.
As partial progress along these lines, the authors recently established that the singular values of random subensembles of the Paley equiangular tight frame obey a Kesten--McKay law~\cite{magsino19}.

The goal of this paper is to describe a promising approach to find new upper bounds on the clique numbers of Paley graphs.  In Section 2, we recall a semidefinite programming approach of Lov\'{a}sz~\cite{lovasz79} that yields bounds on the clique numbers of arbitrary graphs.  By passing to an appropriate subgraph of $G_p$, we show that this produces numerical bounds on $\omega(G_p)$ that usually coincide with the Hanson--Petridis bound and sometimes improve upon it.  In Section 3, we show how to compute these bounds by linear programming, extending the range in which we are able to produce computational evidence.  In Section 4, we derive the relevant dual program and summarize how one might use weak duality to prove a new bound on $\omega(G_p)$ by constructing appropriate number-theoretic functions; see \cref{prop:paleylemma}.  We then conclude with suggestions for future work in this direction.

\section{SEMIDEFINITE PROGRAMMING BOUNDS}
For a graph $G = (V,E)$, its \textbf{complement} $\overline{G}$ is the graph on vertices $V$ with edges $\binom{V}{2} \setminus E$.  An \textbf{isomorphism} between graphs $G = (V,E)$ and $G' = (V',E')$ is a bijection $\varphi\colon V \rightarrow V'$ such that $E$ contains an edge between $v,w \in E$ if and only if $E'$ contains an edge between $\varphi(v),\varphi(w) \in V'$, and an \textbf{automorphism} is an isomorphism between $G$ and itself.  When $G$ is isomorphic to $\overline{G}$, we say that $G$ is \textbf{self-complementary}.  We say $G = (V,E)$ is \textbf{vertex-transitive} if, for every pair of vertices $v,w \in V$, there exists an automorphism $\varphi$ of $G$ with $\varphi(v) = \varphi(w)$.

In the sequel, we label the vertices of every graph $G$ on $n$ vertices by $\mathbf{Z}_n = \{0, 1, \ldots, n - 1\}$ and similarly index the rows and columns of matrices $X \in \mathbf{R}^{n \times n}$ by $\mathbf{Z}_n$ with addition considered modulo $n$.  The \textbf{Lov\'{a}sz theta number} for a graph $G$ on $n$ vertices is defined by the semidefinite program
\[
\vartheta_{\operatorname{L}}(G)
\quad
:=
\quad
\operatorname{max}  
\quad
\displaystyle \sum_{j = 0}^{n-1} \sum_{k = 0}^{n-1} X_{jk}
\quad
\text{s.t.}
\quad
\operatorname{Tr}X=1,
\quad
X_{jk}=0 ~~\forall\{j,k\}\in E(G),
\quad
X\succeq0.
\]
Lov\'{a}sz~\cite{lovasz79} proved the following.

\begin{prop}\label{prop:theta properties}
Let $G$ be any graph on $n$ vertices.
\begin{itemize}
\item[(i)] $\omega(G) \leq \vartheta_{\operatorname{L}}(\overline{G})$.
\item[(ii)] If $G$ is vertex-transitive, then $\vartheta_{\operatorname{L}}(G)\vartheta_{\operatorname{L}}(\overline{G}) = n$.
\end{itemize}
\end{prop}
\begin{proof}
For (i), suppose $C \subseteq \mathbf{Z}_n$ is a maximal clique in $G$, consider the indicator function $1_C\colon \mathbf{Z}_n \rightarrow \{0,1\}$ as a column vector in $\mathbf{R}^n$ indexed by $\mathbf{Z}_n$, and put $X := \frac{1}{\omega(G)} 1_C 1_C^T$.  Then $X$ is feasible in the program $\vartheta_{\operatorname{L}}(\overline{G})$.
Counting the nonzero entries of $X$ then gives
\[
\omega(G) = \sum_{j = 0}^{n - 1} \sum_{k = 0}^{n - 1} X_{jk} \leq \vartheta_{\operatorname{L}}(\overline{G}).
\]
The proof of (ii) is more involved; see Theorem 8 in Lov\'{a}sz~\cite{lovasz79}.
\end{proof}

As a consequence of \cref{prop:theta properties}, every self-complementary vertex-transitive graph $G$ on $n$ vertices satisfies $\omega(G) \leq \vartheta_{\operatorname{L}}(G) = \sqrt{n}$.  This is enough to recover the well-known bound of $\omega(G_p) \leq \sqrt{p}$.  Indeed, to show that $G_p$ is self-complementary, fix any nonzero quadratic non-residue $s \in \mathbf{F}_p^* \setminus Q_p$ and consider the bijection $\mu \colon \mathbf{F}_p \rightarrow \mathbf{F}_p$ defined by $\mu(x) := sx$.  Then for all $v,w \in \mathbf{F}_p$, we have $v - w \in Q_p$ if and only if $\mu(v) - \mu(w) = s(v - w) \not \in Q_p$.  That is, $\mu$ is an isomorphism between $G_p$ and $\overline{G}_p$.  To show that $G_p$ is vertex-transitive, let $a,b \in \mathbf{F}_p$ and consider the map $\tau \colon \mathbf{F}_p \rightarrow \mathbf{F}_p$ defined by $\tau(x) := x - a + b$.  Clearly $\tau(a) = b$.  To see that $\tau$ is an automorphism of $G_p$, it suffices to observe that for any two vertices $v,w \in \mathbf{F}_p$, $\tau(v) - \tau(w) = v - w$.  Hence, $\omega(G_p) \leq \vartheta_{\operatorname{L}}(G_p) = \sqrt{p}$ follows from \cref{prop:theta properties}.

We can improve upon this bound by focusing our attention to the neighborhood of $0$ in $G_p$, namely, the set $Q_p$ of quadratic residues.  Let $L_p$ denote the subgraph of $G_p$ induced by $Q_p$.  Since $G_p$ is vertex-transitive, there exists a maximal clique of $G_p$ containing the vertex $0$.  In particular, $\omega(L_p) = \omega(G_p) - 1$.  By \cref{prop:theta properties}(i), we conclude that
\begin{equation}\label{eq:thetaLpbound}
\omega(G_p) \leq \vartheta_{\operatorname{L}} (\overline{L}_p) + 1.
\end{equation}

For a graph $G$ on $n$ vertices, Schrijver~\cite{schrijver79} proposed strengthening $\vartheta_{\operatorname{L}}(G)$ to
\[
\vartheta_{\operatorname{LS}}(G)
\quad
:=
\quad
\operatorname{max}  
\quad
\displaystyle \sum_{j = 0}^{n-1} \sum_{k = 0}^{n-1} X_{jk}
\quad
\text{s.t.}
\quad
\operatorname{Tr}X=1,
\quad
X_{jk}=0 ~~\forall\{j,k\}\in E(G),
\quad
X\succeq0,
\quad
X\geq0,
\]
where $X\geq0$ denotes entrywise nonnegativity. Clearly $\vartheta_{\operatorname{LS}}(G) \leq \vartheta_{\operatorname{L}}(G)$, and the proof of \cref{prop:theta properties}(i) further establishes establishes $\omega(G) \leq \vartheta_{\operatorname{LS}}(G)$. This strengthening leads to the bound
\begin{equation}\label{eq:thetaLSbound}
\omega(G_p) \leq \vartheta_{\operatorname{LS}} (\overline{L}_p) + 1.
\end{equation}
To compare the bounds \eqref{eq:thetaLpbound} and \eqref{eq:thetaLSbound} to the Hanson--Petridis bound \eqref{eq:bestknown}, we set
\[
\operatorname{HP}(p) := \dfrac{\sqrt{2p - 1} + 1}{2}, 
\qquad
\operatorname{L}(p) := \vartheta_{\operatorname{L}} (\overline{L}_p) + 1,
\qquad
\operatorname{LS}(p) := \vartheta_{\operatorname{LS}} (\overline{L}_p) + 1,
\]
and we compare these in the range $p < 3000$ in \cref{fig:p3000} and \cref{table:p3000}.    We observe $\lfloor{\operatorname{L}(p)}\rfloor = \lfloor \operatorname{LS}(p) \rfloor = \lfloor \operatorname{HP}(p) \rfloor$ for most primes in this range, providing an equivalent upper bound on $\omega(G_p)$.  Interestingly, $\lfloor \operatorname{LS}(p) \rfloor = \lfloor \operatorname{HP}(p) \rfloor - 1$ for 17 values of $p < 3000$.  

Gvozdenovi\'{c}, Laurent and Vallentin\cite{gvoz09} used semidefinite programming to compute several values of $\operatorname{L}(p)$, which in their notation is $N_+(\operatorname{TH}(P_p))$.  For instance, they compute $\operatorname{L}(809)$ in 4.5 hours on a 3GHz  processor with 1GB of RAM.  They further introduced the so-called block-diagonal hierarchy of semidefinite programs, which allowed them to compute sharper bounds than $\operatorname{L}(p)$ somewhat more efficiently in the range $p \leq 809$.  In order to compute numerical values of $\operatorname{L}(p)$ and $\operatorname{LS}(p)$ efficiently, we leverage the symmetry of $L_p$ to reformulate both $\vartheta_{\operatorname{L}} (\overline{L}_p)$ and $\vartheta_{\operatorname{LS}} (\overline{L}_p)$ as linear programs in the next section.  Using this approach on a 3.4GHz processor with 8GB of RAM, we compute $\operatorname{L}(809)$ in under 20 seconds.

\begin{figure}
    \centering
    \includegraphics[width=3in]{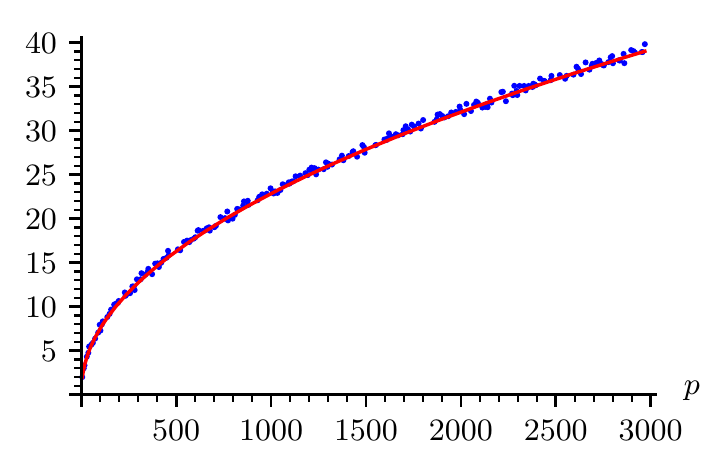}
    \hspace{1em}
    \includegraphics[width=3in]{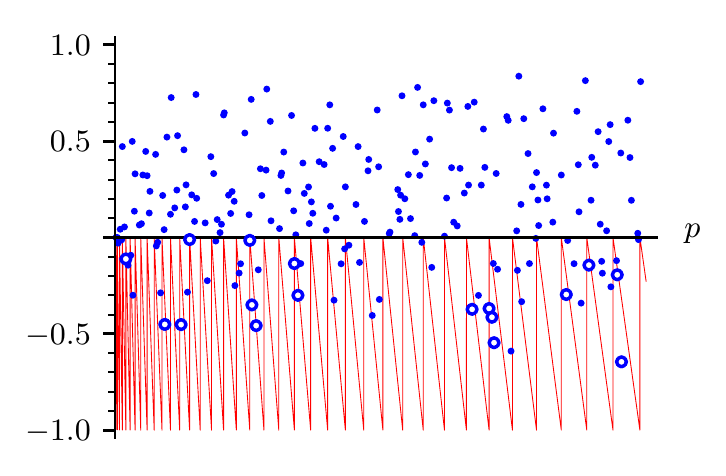}
    \caption{Comparison of $\operatorname{LS}(p)$ and $\operatorname{HP}(p)$ for the 211 primes $p \equiv 1~(\operatorname{mod}4)$ with $p < 3000$.  For 60 such primes, $\operatorname{LS}(p) \leq \operatorname{HP}(p)$.  For 17 such primes, $\operatorname{LS}(p) < \lfloor \operatorname{HP}(p) \rfloor$.
    \textbf{(left)} The blue points $(p,\operatorname{LS}(p))$ appear to concentrate around the red curve $y = \operatorname{HP}(p)$.
    \textbf{(right)} If a point $(p,\operatorname{LS}(p) - \operatorname{HP}(p))$ lies below the red curve $y = \lfloor \operatorname{HP} (p) \rfloor - \operatorname{HP}(p)$, then $\operatorname{LS}(p) < \lfloor \operatorname{HP}(p) \rfloor$, in which case we plot the point as a circle.  }
    \label{fig:p3000}
\end{figure}

\section{REDUCTION TO LINEAR PROGRAMMING}

Recall that a matrix $X \in \mathbf{R}^{n \times n}$ is \textbf{circulant} if $X_{j+1,k+1} = X_{jk}$ for all $j,k \in \mathbf{Z}_n$.  A graph $G$ is said to be circulant if there exists a labeling of its vertices such that its adjacency matrix is circulant.  We note that for every prime $p$, the graph $L_p$ is circulant.  Indeed, select a generator $\alpha$ of the mulitplicative subgroup $Q_p$, and order the elements of $Q_p$ as $1, \alpha, \ldots, \alpha^{n - 1}$.  Then $L_p$ is circulant since $\alpha^j - \alpha^k \in Q_p$ if and only if $\alpha^{j + 1} - \alpha^{k + 1} \in Q_p$.  Since the complement of a circulant graph is also circulant, it holds that $\overline{L}_p$ is circulant as well.

Schrijver~\cite{schrijver79} showed that the semidefinite programming formulations of both $\vartheta_{\operatorname{L}}$ and $\vartheta_{\operatorname{LS}}$ can be reduced to linear programs for certain classes of graphs.  In order to state one such linear programming formulation, we take the Fourier transform of $f \colon \mathbf{Z}_n \rightarrow \mathbf{C}$ to be the function $\widehat{f} \colon \mathbf{Z}_n \rightarrow \mathbf{C}$ defined by
\[
\widehat{f}(k) := \sum_{j = 0}^{n - 1} f(j) e^{-2\pi i jk / n}.
\]

\begin{prop}\label{prop: thetaL LP}
Let $G$ be any circulant graph with vertex set $\mathbf{Z}_n$.  Then
\begin{equation}
\label{eq:thetaLLP}
\vartheta_{\operatorname{L}}(G) 
\quad
=
\quad
\operatorname{max}
\quad
n \displaystyle \sum_{k = 0}^{n - 1} f(k)
\quad
\text{s.t.}
\quad
f(0) = \frac{1}{n},
\quad
f(k) = 0~~\forall\{0,k\}\in E(G),
\quad
\widehat{f} \geq 0.
\end{equation}
\end{prop}
\begin{proof}
Let $\vartheta_{\operatorname{LLP}}(G)$ denote the right-hand side of $\eqref{eq:thetaLLP}$.
First, we show that $\vartheta_{\operatorname{LLP}}(G)\leq\vartheta_{\operatorname{L}}(G)$.
Take any feasible $f$ in the program $\vartheta_{\operatorname{LLP}}(G)$ and consider the circulant matrix $X$ defined by $X_{0k}:=f(k)$.
Then $\operatorname{Tr}(X) = 1$ follows from $f(0) = 1/n$, the edge constraints on $X$ follow from the edge constraints on $f$, and since the eigenvalues of $X$ are the Fourier coefficients $\{\widehat{f}(k) : k \in \mathbf{Z}_n\}$, we see that $X\succeq0$ follows from $\widehat{f}\geq0$.
Furthermore, $\sum_{j = 0}^{n - 1}\sum_{k = 0}^{n - 1} X_{jk}=n\sum_{k=0}^{n-1}f(k)$.
Since every feasible point in $\vartheta_{\operatorname{LLP}}(G)$ can be mapped to a feasible point in $\vartheta_{\operatorname{L}}(G)$ with the same value, we conclude that $\vartheta_{\operatorname{LLP}}(G)\leq\vartheta_{\operatorname{L}}(G)$.

For the other direction, fix any $X^{(0)}$ that is feasible in $\vartheta_{\operatorname{L}}(G)$, and for each $\ell \in \mathbf{Z}_n$, consider the matrix $X^{(\ell)} \in \mathbf{R}^{n \times n}$ defined by $X^{(\ell)}_{jk} := X^{(0)}_{j + \ell, k + \ell}$.  Then $X^{(\ell)}$ is also feasible in $\vartheta_{\operatorname{L}}(G)$ with the same value:
\[
\sum_{j = 0}^{n - 1}\sum_{k = 0}^{n - 1} X^{(\ell)}_{jk}
=\sum_{j = 0}^{n - 1}\sum_{k = 0}^{n - 1} X^{(0)}_{j + \ell, k + \ell}
=\sum_{j = 0}^{n - 1}\sum_{k = 0}^{n - 1} X^{(0)}_{jk}.  
\]
Averaging over this orbit produces a circulant matrix $X := \frac{1}{n} \sum_{\ell = 0}^{n - 1} X^{(\ell)}$ that, by convexity, is also feasible in $\vartheta_{\operatorname{L}}(G)$, and that, by linearity, has the same value.
Take $f \colon \mathbf{Z}_n \rightarrow \mathbf{R}$ defined by $f(k) := X_{0k}$ to obtain a feasible point in $\vartheta_{\operatorname{LLP}}(G)$ with the same value.
This implies the reverse inequality $\vartheta_{\operatorname{L}}(G)\leq\vartheta_{\operatorname{LLP}}(G)$.
\end{proof}

Arguing similarly establishes the following.

\begin{prop}\label{prop: thetaLS LP}
Let $G$ be any circulant graph with vertex set $\mathbf{Z}_n$.  Then
\begin{equation}
\label{eq:thetaLSLP}
\vartheta_{\operatorname{LS}}(G) 
\quad
=
\quad
\operatorname{max}
\quad
n \displaystyle \sum_{k = 0}^{n - 1} f(k)
\quad
\text{s.t.}
\quad
f(0) = \frac{1}{n},
\quad
f(k) = 0~~\forall\{0,k\}\in E(G),
\quad
\widehat{f} \geq 0,
\quad
f \geq 0.
\end{equation}
\end{prop}

We used the linear program formulations in \cref{prop: thetaL LP,prop: thetaLS LP} to compute the values of $\operatorname{L}(p)$ and $\operatorname{LS}(p)$ reported in \cref{fig:p3000} and \cref{table:p3000}.

\section{DUAL CERTIFICATES}

In this section, we derive the dual program of $\vartheta_{\operatorname{LS}}(G)$ for arbitrary circulant graphs $G$.
Since every feasible point of the dual program of $\vartheta_{\operatorname{LS}}(\overline{L}_p)$ gives an upper bound on $\omega(G_p)$, this section might allow one to prove a new closed-form upper bound on $\omega(G_p)$.
Recall that for a closed convex cone $K \subseteq \mathbf{R}^n$, its \textbf{dual cone} is given by
\[
K^* := \{y \in \mathbf{R}^n : \langle x,y \rangle \geq 0 \text{ for all } x \in K\}.
\]
Given closed convex cones $K,M \subseteq \mathbf{R}^n$, the primal program
\begin{equation}
\label{eq:standardPrimal}
\operatorname{max}
\quad
\langle c,x \rangle
\quad
\text{s.t.}
\quad
b - Ax \in K,
\quad
x \in M
\end{equation}
has the corresponding dual program
\[
\operatorname{min}
\quad
\langle b,y \rangle
\quad
\text{s.t.}
\quad
A^Ty - c \in M^*,
\quad
y \in K^*.
\]
We will use the above formulation to derive a relatively clean expression for the dual program of \eqref{eq:thetaLSLP}.

\begin{prop}\label{prop: thetaLS dual}
Let $G$ be any circulant graph with vertex set $\mathbf{Z}_n$. Then
\[
\vartheta_{\operatorname{LS}}(G)
\quad
=
\quad
\operatorname{min}
\quad
f(0)
\quad
\text{s.t.} 
\quad
f(k) = 0~~\forall\{0,k\}\in E(\overline{G}),
\quad
f \geq g + 1,
\quad
\widehat{g} \geq 0.
\]
\end{prop}

\begin{proof}
By strong duality, it suffices to show that the right-hand side is the dual program of \eqref{eq:thetaLSLP}.
To this end, we first write the linear program \eqref{eq:thetaLSLP} in the form \eqref{eq:standardPrimal}.  We identify functions $f\colon\mathbf{Z}_n \rightarrow \mathbf{R}$ with column vectors in $\mathbf{R}^n$ indexed by $\mathbf{Z}_n$.  Let $P$ denote the projection operator defined by
\[
(Pf)(k) 
:= \left\{\begin{array}{cl} 0 &  \text{if } \{0,k\}\in E(\overline{G}) \\ f(k) & \text{otherwise.} \end{array}\right.
\]
Then $f(0)=\frac{1}{n}$ and $f(k) = 0$ for every $\{0,k\}\in E(G)$ if and only if $Pf = \frac{1}{n}\delta_0$.
Next, let $R$ denote the reversal operator defined by $(Rf)(k) := f(-k)$, and let $C$ denote the cosine transform defined by
\[
(Cf)(k) 
:= \sum_{j = 0}^{n - 1} f(j) \cos(2\pi j k / n).
\]
Then $f\in\mathbf{R}^n$ satisfies $\widehat{f} \geq 0$ if and only if $Rf = f$ and $Cf \geq 0$.
Overall, \eqref{eq:thetaLSLP} is equivalent to
\[
\vartheta_{\operatorname{LS}}(G)
\quad
=
\quad
\operatorname{max}
\quad
\langle n \mathbf{1}, f \rangle
\quad
\text{s.t.}
\quad
\left[\begin{array}{c}\frac{1}{n} \delta_0 \\ 0 \\ 0\end{array}\right] - \left[\begin{array}{c} P \\ I-R \\ -C \end{array}\right]f \in \{0 \in \mathbf{R}^n\} \times \{0 \in \mathbf{R}^n\} \times \mathbf{R}^n_{\geq 0},
\quad
f \in \mathbf{R}_{\geq 0}^n,
\]
where $\mathbf{R}_{\geq 0}$ denotes the set of nonnegative real numbers.
Since $(\{0 \in \mathbf{R}^n\} \times \{0 \in \mathbf{R}^n\} \times \mathbf{R}^n_{\geq 0})^*=\mathbf{R}^n\times\mathbf{R}^n\times\mathbf{R}_{\geq 0}^n$, the dual program is given by following, written in terms of dual variables $y=(u,v,w)\in(\mathbf{R}^n)^3$:
\begin{equation}\label{eq:dual1}
\operatorname{min}
\quad
\frac{1}{n} u(0)
\quad
\text{s.t.}
\quad
Pu + (I-R)v - Cw - n\mathbf{1} \in \mathbf{R}^n_{\geq 0},
\quad
w \in \mathbf{R}^n_{\geq 0}.
\end{equation}
Since $G$ is a circulant graph, we see that $\{0,k\}\in E(G)$ precisely when $\{0,-k\}\in E(G)$, and so $RP = PR$.
Also, $RC = CR$.
We apply these facts to observe that $(u,v,w)$ is feasible in~\eqref{eq:dual1} if and only if $(Ru,-v,Rw)$ is feasible in~\eqref{eq:dual1}, and with the same value.
Indeed, $R$ maps $\mathbf{R}^n_{\geq 0}$ to itself, and
\[
P(Ru)+(I-R)(-v)-C(Rw)-n\mathbf{1}
= R(Pu + (I-R)v - Cw - n\mathbf{1}),
\qquad
\frac{1}{n}(Ru)(0)=\frac{1}{n}u(0).
\]
By averaging these two feasible points, we obtain the following equivalent program:
\begin{equation}\label{eq:dual2}
\operatorname{min}
\quad
\frac{1}{n} u(0)
\quad
\operatorname{s.t.}
\quad
Ru = u,
\quad
Rw = w,
\quad
Pu \geq Cw + n\mathbf{1},
\quad
w \geq0.
\end{equation}
At this point, we may relax the constraint $Ru = u$ since $Cw + n\mathbf{1}$ is even.
Also, $w\in\mathbf{R}^n$ satisfies $Rw = w$ if and only if $Cw = \widehat{w}$.
Changing variables to $f = \frac{1}{n} Pu$ and $g = \frac{1}{n}\widehat{w}$ then gives the result.
\end{proof}

As such, given a circulant graph $G$, any $(f,g)$ that is feasible in the corresponding linear program in \cref{prop: thetaLS dual} yields an upper bound on $\vartheta_{\operatorname{LS}}(G)$.
Recalling \eqref{eq:thetaLSbound}, we now specialize to the case of Paley graphs:

\begin{prop}\label{prop:paleylemma}
Given a prime $p \equiv 1~(\operatorname{mod}4)$, let $\alpha$ denote a generator of the multiplicative group $Q_p$ of quadratic residues modulo $p$, and set $n = (p - 1)/2$. Suppose that $f,g\colon \mathbf{Z}_n \rightarrow \mathbf{R}$ together satisfy
\begin{itemize}
\item[(i)]
$f(k) = 0$ for every $k \in \mathbf{Z}_n$ with $\alpha^k - 1 \in Q_p$,
\item[(ii)]
$f \geq g + 1$, and
\item[(iii)]
$\widehat{g} \geq 0$.
\end{itemize}
Then $\omega(G_p) \leq f(0) + 1$.
\end{prop}

Arguing similarly to \cref{prop: thetaLS dual} gives a comparable dual program for $\vartheta_{\operatorname{L}}(G)$.  In fact, the resulting program corresponds to adding the constraint $f = g + 1$ to the program in \cref{prop: thetaLS dual}.  Considering the numerical data in \cref{table:p3000}, we expect these bounds to match frequently.

\section{FUTURE WORK}
In this paper, we used linear programming to find numerical upper bounds on $\omega(G_p)$ that usually match and sometimes improve on the Hanson--Petridis bound $\lfloor \operatorname{HP}(p) \rfloor$.  Our experiments suggest the following.

\begin{conj}\label{conj:wewin}
For infinitely many primes $p \equiv 1~(\operatorname{mod}4)$, it holds that $\operatorname{LS}(p) < \lfloor \operatorname{HP}(p) \rfloor$.
\end{conj}

With appropriate number-theoretic functions, one might use \cref{prop:paleylemma} to prove \cref{conj:wewin}.
This pursuit of ``magic functions'' bears some resemblance to recent progress in sphere packing; see Cohn~\cite{cohn17} for a survey.
We note that Gvozdenovi\'{c}, Laurent and Vallentin\cite{gvoz09} introduced a semidefinite programming hierarchy that gives numerical bounds on $\omega(G_p)$ that are sharper than $\operatorname{HP}(p)$ in the range $p \leq 809$.  However, these semidefinite programs are still rather slow.
For our linear programming computations, we used \texttt{GLPK} within \texttt{SageMath}~\cite{sage}.  We believe that our code could be sped up significantly by incorporating the fast Fourier transform\cite{vanderbei12}, possibly giving new bounds on $\omega(G_p)$ for significantly larger primes $p$.

\acknowledgments 
 
MM and DGM were partially supported by AFOSR FA9550-18-1-0107.
DGM was also supported by NSF DMS 1829955 and the Simons Institute of the Theory of Computing.

\bibliography{report} 

\begin{thebibliography}{10}

\bibitem{chung89}
Chung, F. R.~K., Graham, R.~L., and Wilson, R.~M., ``Quasi-random graphs,''
  {\em Combinatorica}~{\bf 9}(4),  345--362 (1989).

\bibitem{bollobas01}
Bollob\'{a}s, B.,  [{\em Random graphs}{\nolinebreak\hspace{0.1em}]}, vol.~73
  of {\em Cambridge Studies in Advanced Mathematics}, Cambridge University
  Press, Cambridge, second~ed. (2001).

\bibitem{strohmer03}
Strohmer, T. and Heath~Jr, R.~W., ``Grassmannian frames with applications to
  coding and communication,'' {\em Applied and computational harmonic
  analysis}~{\bf 14}(3),  257--275 (2003).

\bibitem{renes07}
Renes, J.~M., ``Equiangular tight frames from {P}aley tournaments,'' {\em
  Linear Algebra and its Applications}~{\bf 426}(2-3),  497--501 (2007).

\bibitem{waldron09}
Waldron, S., ``On the construction of equiangular frames from graphs,'' {\em
  Linear Algebra and its applications}~{\bf 431}(11),  2228--2242 (2009).

\bibitem{bandeira13}
Bandeira, A.~S., Fickus, M., Mixon, D.~G., and Wong, P., ``The road to
  deterministic matrices with the restricted isometry property,'' {\em Journal
  of Fourier Analysis and Applications}~{\bf 19}(6),  1123--1149 (2013).

\bibitem{bandeira17}
Bandeira, A.~S., Mixon, D.~G., and Moreira, J., ``A conditional construction of
  restricted isometries,'' {\em International Mathematics Research
  Notices}~{\bf 2017}(2),  372--381 (2017).

\bibitem{cohen88}
Cohen, S.~D., ``Clique numbers of {P}aley graphs,'' {\em Quaestiones
  Math.}~{\bf 11}(2),  225--231 (1988).

\bibitem{erdos35}
Erd\H{o}s, P. and Szekeres, G., ``A combinatorial problem in geometry,'' {\em
  Compositio Math.}~{\bf 2},  463--470 (1935).

\bibitem{graham90}
Graham, S.~W. and Ringrose, C.~J., ``Lower bounds for least quadratic
  nonresidues,'' in [{\em Analytic number theory ({A}llerton {P}ark, {IL},
  1989)}{\nolinebreak\hspace{0.1em}]},  {\em Progr. Math.} {\bf 85},  269--309,
  Birkh\"{a}user Boston, Boston, MA (1990).

\bibitem{hanson19}
{Hanson}, B. and {Petridis}, G., ``Refined estimates concerning sumsets
  contained in the roots of unity,'' {\em arXiv:1905.09134}  (2019).

\bibitem{maistrelli06}
Maistrelli, E. and Penman, D.~B., ``Some colouring problems for {P}aley
  graphs,'' {\em Discrete Math.}~{\bf 306}(1),  99--106 (2006).

\bibitem{bachoc13}
Bachoc, C., Matolcsi, M., and Ruzsa, I.~Z., ``Squares and difference sets in
  finite fields,'' {\em Integers}~{\bf 13},  Paper No. A77, 5 (2013).

\bibitem{shearer86}
Shearer, J.~B., ``Lower bounds for small diagonal {R}amsey numbers,'' {\em J.
  Combin. Theory Ser. A}~{\bf 42}(2),  302--304 (1986).

\bibitem{exooData}
Exoo, G., ``Independence numbers for {P}aley graphs.''
\newblock {\tt http://isu.indstate.edu/ge/PALEY/index.html}.

\bibitem{taoBlog}
Tao, T., ``Open question: deterministic {UUP} matrices.''
\newblock \newline{\tt
  http://terrytao.wordpress.com/2007/07/02/open-question-deterministic-uup-matrices/}.

\bibitem{bourgain11}
Bourgain, J., Dilworth, S., Ford, K., Konyagin, S., Kutzarova, D., et~al.,
  ``Explicit constructions of {RIP} matrices and related problems,'' {\em Duke
  Mathematical Journal}~{\bf 159}(1),  145--185 (2011).

\bibitem{mixon15}
Mixon, D.~G., ``Explicit matrices with the restricted isometry property:
  Breaking the square-root bottleneck,'' in [{\em Compressed sensing and its
  applications}{\nolinebreak\hspace{0.1em}]},   389--417, Springer (2015).

\bibitem{magsino19}
Magsino, M., Mixon, D.~G., and Parshall, H., ``Kesten--{M}c{K}ay law for random
  subensembles of {P}aley equiangular tight frames,'' {\em arXiv preprint
  arXiv:1905.04360}  (2019).

\bibitem{lovasz79}
Lov\'{a}sz, L., ``On the {S}hannon capacity of a graph,'' {\em IEEE Trans.
  Inform. Theory}~{\bf 25}(1),  1--7 (1979).

\bibitem{schrijver79}
Schrijver, A., ``A comparison of the {D}elsarte and {L}ov\'{a}sz bounds,'' {\em
  IEEE Trans. Inform. Theory}~{\bf 25}(4),  425--429 (1979).

\bibitem{gvoz09}
Gvozdenovi\'{c}, N., Laurent, M., and Vallentin, F., ``Block-diagonal
  semidefinite programming hierarchies for 0/1 programming,'' {\em Oper. Res.
  Lett.}~{\bf 37}(1),  27--31 (2009).

\bibitem{cohn17}
Cohn, H., ``A conceptual breakthrough in sphere packing,'' {\em Notices Amer.
  Math. Soc.}~{\bf 64}(2),  102--115 (2017).

\bibitem{sage}
{The Sage Developers}, {\em {S}ageMath, the {S}age {M}athematics {S}oftware
  {S}ystem ({V}ersion 8.7)} (2019).
\newblock {\tt https://www.sagemath.org}.

\bibitem{vanderbei12}
Vanderbei, R.~J., ``Fast {F}ourier optimization: sparsity matters,'' {\em Math.
  Program. Comput.}~{\bf 4}(1),  53--69 (2012).

\end{thebibliography}
\bibliographystyle{spiebib} 

\begin{table}
\centering
\caption{Comparison of $\omega(G_p)$ with the upper bounds $\operatorname{HP}(p), \operatorname{L}(p)$ and $\operatorname{LS}(p)$ for the 63 primes $p \equiv 1~(\operatorname{mod}4)$ with $p < 3000$ and $\lfloor\operatorname{HP}(p)\rfloor~\neq~\lfloor\operatorname{LS}(p)\rfloor$.  For the 148 unlisted primes $p \equiv 1~(\operatorname{mod}4)$ with $p < 3000$, we observed $\lfloor\operatorname{HP}(p)\rfloor = \lfloor\operatorname{LS}(p)\rfloor$.}
\label{table:p3000}
\vspace{1em}
{\def\arraystretch{1.1}
\begin{tabular}{rrrrr}
\hline
 $p$ & $\omega(G_p)$ & $\operatorname{HP}(p)$ & $\operatorname{L}(p)$ & $\operatorname{LS}(p)$\\
 \hline
61 & 5 & 6.0000 & \textbf{5.9009} & \textbf{5.8886} \\
109 & 6 & \textbf{7.8655} & 8.0070 & 8.0018 \\
173 & 8 & \textbf{9.7871} & 10.3165 & 10.2339 \\
281 & 7 & 12.3427 & \textbf{11.9023} & \textbf{11.8916} \\
293 & 8 & \textbf{12.5934} & 13.1270 & 13.1145 \\
353 & 9 & \textbf{13.7759} & 14.4454 & 14.3045 \\
373 & 8 & 14.1473 & \textbf{13.7229} & \textbf{13.6952} \\
421 & 9 & 15.0000 & 15.0253 & \textbf{14.9892} \\
457 & 11 & \textbf{15.6079} & 16.3859 & 16.3503 \\
541 & 11 & \textbf{16.9393} & 17.4222 & 17.3589 \\
673 & 11 & \textbf{18.8371} & 19.0862 & 19.0251 \\
733 & 11 & \textbf{19.6377} & 20.3389 & 20.1800 \\
757 & 11 & \textbf{19.9487} & 20.1284 & 20.0668 \\
761 & 11 & 20.0000 & 20.0297 & \textbf{19.9851} \\
773 & 11 & 20.1532 & \textbf{19.8771} & \textbf{19.8033} \\
797 & 9 & 20.4562 & 20.1191 & \textbf{19.9988} \\
821 & 12 & \textbf{20.7546} & 21.3005 & 21.1115 \\
829 & 11 & \textbf{20.8531} & 21.1864 & 21.0711 \\
877 & 13 & \textbf{21.4344} & 22.2406 & 22.0372 \\
997 & 13 & \textbf{22.8215} & 23.5064 & 23.4550 \\
1009 & 11 & \textbf{22.9555} & 23.2465 & 23.0941 \\
1013 & 11 & 23.0000 & 23.0713 & \textbf{22.8647} \\
1033 & 11 & 23.2211 & 23.0159 & \textbf{22.9210} \\
1093 & 12 & \textbf{23.8720} & 24.2033 & 24.1343 \\
1181 & 12 & \textbf{24.7951} & 25.2438 & 25.1739 \\
1289 & 15 & \textbf{25.8821} & 26.4445 & 26.4064 \\
1373 & 14 & \textbf{26.6964} & 27.3171 & 27.1684 \\
1481 & 15 & \textbf{27.7075} & 28.5703 & 28.3694 \\
1489 & 13 & \textbf{27.7809} & 28.3456 & 28.1480 \\
1597 & 13 & \textbf{28.7533} & 29.0803 & 29.0021 \\
1613 & 14 & \textbf{28.8945} & 29.2067 & 29.1143 \\
1621 & 13 & \textbf{28.9649} & 29.8909 & 29.7006 \\
\end{tabular}
\hspace{1em}
\begin{tabular}{rrrrr}
\hline
 $p$ & $\omega(G_p)$ & $\operatorname{HP}(p)$ & $\operatorname{L}(p)$ & $\operatorname{LS}(p)$\\
 \hline
1697 & 13 & \textbf{29.6247} & 30.1311 & 30.0687 \\
1709 & 13 & \textbf{29.7276} & 30.6383 & 30.5067 \\
1721 & 13 & \textbf{29.8300} & 30.2173 & 30.1523 \\
1801 & 14 & \textbf{30.5042} & 31.3490 & 31.2143 \\
1949 & 14 & \textbf{31.7130} & 32.1343 & 32.0719 \\
1973 & 13 & \textbf{31.9046} & 32.3272 & 32.1354 \\
2017 & 13 & 32.2530 & \textbf{31.9977} & \textbf{31.8802} \\
2029 & 14 & \textbf{32.3473} & 33.3049 & 33.0499 \\
2081 & 14 & \textbf{32.7529} & 33.5041 & 33.3159 \\
2089 & 14 & \textbf{32.8149} & 33.3957 & 33.1789 \\
2113 & 13 & 33.0000 & \textbf{32.9818} & \textbf{32.6315} \\
2129 & 13 & 33.1228 & \textbf{32.8782} & \textbf{32.7089} \\
2141 & 13 & 33.2147 & \textbf{32.7483} & \textbf{32.6685} \\
2213 & 15 & \textbf{33.7603} & 34.5759 & 34.3880 \\
2221 & 15 & \textbf{33.8204} & 34.5585 & 34.4287 \\
2281 & 17 & \textbf{34.2676} & 35.2916 & 35.1050 \\
2309 & 15 & \textbf{34.4743} & 35.2676 & 35.0910 \\
2333 & 14 & \textbf{34.6504} & 35.1840 & 35.0862 \\
2357 & 15 & \textbf{34.8256} & 35.2750 & 35.0886 \\
2477 & 15 & \textbf{35.6888} & 36.4402 & 36.2304 \\
2549 & 15 & 36.1966 & 36.0154 & \textbf{35.9006} \\
2609 & 15 & \textbf{36.6144} & 37.5661 & 37.2694 \\
2617 & 15 & \textbf{36.6697} & 37.2249 & 37.0475 \\
2657 & 15 & \textbf{36.9452} & 37.9459 & 37.7595 \\
2677 & 15 & 37.0821 & 37.1579 & \textbf{36.9388} \\
2789 & 15 & \textbf{37.8397} & 38.6001 & 38.3378 \\
2797 & 15 & \textbf{37.8932} & 38.5759 & 38.4791 \\
2837 & 15 & 38.1597 & 38.1846 & \textbf{37.9661} \\
2861 & 16 & 38.3186 & \textbf{37.8309} & \textbf{37.6733} \\
2897 & 15 & \textbf{38.5559} & 39.4579 & 39.1647 \\
2909 & 15 & \textbf{38.6346} & 39.2540 & 39.0498 \\
\phantom{2909} & \\
\end{tabular}
}
\end{table}

\end{document}